\documentclass[preprint,12pt]{amsart}

\usepackage{graphicx}
\usepackage{amssymb}
\usepackage{amsthm}
\usepackage{amsmath}
\usepackage{pstricks,pst-node}
\usepackage{pst-plot}
\usepackage{pst-fill}
\usepackage{pst-grad}
\usepackage{pst-coil}
\usepackage[colorinlistoftodos]{todonotes}

\usepackage{lineno}

\usepackage[utf8]{inputenc}
\usepackage[spanish, english]{babel}

\newpsobject{grilla}{psgrid}{subgriddiv=1,griddots=10,gridlabels=6pt}
\newtheorem{theorem}{Theorem}
\newtheorem{lemma}{Lemma}

\begin{document}



\title{Closed and asymptotic formulas for energy of some circulant graphs}

\author[D. Bl\'azquez-Sanz]{David Bl\'azquez-Sanz}
\author[C. Mar\'in]{Carlos Alberto Mar\'in Arango}

\address{Universidad Nacional de Colombia - Sede Medell\'in \hfill\break\indent  Facultad de Ciencias 
\hfill\break\indent Escuela de Matem\'aticas \hfill\break\indent  Medell\'in, Colombia}
\email{dblazquezs@unal.edu.co}

\address{Instituto de Matem\'aticas \hfill\break\indent  Universidad de Antioquia \hfill\break\indent Medell\'in, Colombia}
\email{calberto.marin@matematicas.udea.edu.co}

\maketitle

\begin{abstract}
We consider circulant graphs $G(r,N)$ where the vertices are the integers modulo $N$ and the neighbours of $0$ are \linebreak $\{-r,\ldots,-1,1,\ldots,r\}$. The energy of $G(r,N)$ is a trigonometric sum of $N\times r$ terms. For low values of $r$ we compute this sum explicitly. We also study the asymptotics of the energy of $G(r,N)$ for $N\to\infty$.
There is a known integral formula for the linear growth coefficient, we find a new expression of the form of a finite trigonometric sum with $r$ terms. As an application we show that in the family $G(r,N)$ for $r\leq 4$ there is a finite number of hyperenergetic graphs. On the other hand, for each $r>4$ there is at most a finite number of non-hyperenergetic graphs of the form $G(r,N)$. Finally we show that the graph $G(r,2r+1)$ minimizes the energy among all the regular graphs of degree $2r$.
\end{abstract}


\section{Introduction}
\label{S:1}
Let $G$ be a graph with $N$ vertices and eigenvalues $\lambda_1,\dots,\lambda_N$. 
The energy of $G$ is defined as the sum of the absolute values of its eigenvalues:
$$\mathcal{E}(G) = \sum_{j=1}^N |\lambda_j|.$$
This concept was introduced in the mathematical literature by Gutman in 1978, \cite{Gut-78}. For further details on this theory we refer to \cite{Li-12} and the literature cited in.

The $n$-vertex complete graph $K_n$ has eigenvalues $n-1$ y $-1$ ($(n-1)$ times). Therefore, $\mathcal{E}(K_n)=2(n-1)$. In \cite{Gut-78}, it was conjectured that the complete graph $K_n$ has the largest energy among all $n$-vertex graphs $G$, that is,
$\mathcal{E}(G)\le 2(n-1)$ with equality iff $G=K_n$.
By means of counterexamples, this conjecture was shown to be false, see for instance \cite{walikar-01}.  

An $n$-vertex graph $G$ whose energy satisfies $\mathcal{E}(G) > 2(n-1)$ is called {\em{hyperenergetic\/}}. These graphs has been introduced in \cite{Gut-99} in relation with some problems of molecular chemistry. A simplest construction of a family of hyperenergetic graphs is due to Walikar et al, \cite{Walikar-99}, where authors showed that the line graph of $K_n$, $n\ge 5$ is hyperenergetic. There are a number of other recent results on hyperenergetic graphs \cite{Balakrishnan2004, Koolen-2000}. There are also some recent results for the hyperenergetic circulant graphs, see for instance \cite{Shparlinski2006, Stevanovic2005}.

In this paper we consider the following family of circulant graphs. For each $N$ and $r$ with
$r\leq \lfloor \frac{N}{2}\rfloor$ we define the circulant graph $G(r,N)$ with vertices 
$0,1,\ldots,N-1$ in which there is an edge between $m$ and $n$ if the inequation
$|i-j + xN|\leq r$ has a solution $x\in\mathbb Z$, see fig. \ref{figure_circulant}.
We give a closed formula for the energy of the circulant graph $G(r,N)$. Based on that formula we also show that the energy of the graph $K_{2N}-M$ obtaining by deleting the edges of a perfect matching of a complete graph of index $2N$ is non-hyperenegetic. Moreover, we show that the graph $K_{2N}-M$ has minimal energy over the set of regular graphs of degree $2N-2$. Finally, we find some bounds and explicit expressions for $lim_{N\to\infty} \frac{\mathcal{E}\left(G(r,N)\right)}{N-1}$. 
This allows us to conclude that for $r\le 4$ there is finite number of hyperenergetic graphs of the form $G(r,N)$; moreover, for $r\ge 5$ there is a finite number of non-hyperenergetic graphs of the form $G(r,N)$.

\begin{figure}[h!]
\begin{center}
\begin{pspicture}(-1,-1)(8,4)
\rput(2,2){\color{black}\textbullet}
\rput(2,1){\color{black}\textbullet}
\rput(5,2){\color{black}\textbullet}
\rput(5,1){\color{black}\textbullet}
\rput(4,3){\color{black}\textbullet}
\rput(4,0){\color{black}\textbullet}
\rput(3,0){\color{black}\textbullet}
\rput(2,1.5){$\vdots$}
\rput(3.5,3){$\hdots$}
\rput{45}(2.5,2.5){$\dots$}
\rput{315}(2.5,0.5){$\dots$}
\psline(4,3)(5,2)(5,1)(4,0)(3,0)
\psline(4,3)(5,1)
\psline(5,2)(4,0)
\psline(5,2)(3,3)
\psline(5,1)(3,0)
\psline(3,0.5)(4,0)
\end{pspicture}
\caption{$\text{The graph } G(2,N)$}\label{figure_circulant}
\end{center}
\end{figure}
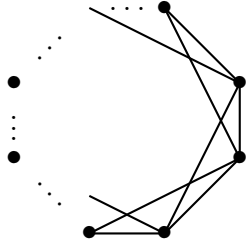

\section{Preliminaries}

Let us remark the following easy to check facts:
\begin{itemize}
\item[(a)] $G(1,N)$ is the $N$-cycle graph.
\item[(b)] $G(r,2r+1)$ is the complete graph $K_{2r+1}$.
\item[(c)] $G(r,2r+2)$ is $K_{2r+2}-M$ the complement of a perfect matching of the edges in a complete graph $K_{2r+1}$.
\item[(d)] $G(r,2r+3)$ is $K_{2r+3}-H$ the complement of a hamiltonian cycle in a complete graph $K_{2r+3}$.
\end{itemize}

The eigenvalues of $G(r,N)$ can be computed by terms of a finite Fourier transform, it turns out that if $\lambda(r,N,k)$ denotes the $k$-th eigenvalue of $G(r,N)$ then;
$$\lambda(r,N,k) = u\left(r,\frac{2k\pi}{N}\right)$$
where
$$u(r,\theta) = 2\sum_{m=1}^r\cos(m\theta).$$
Thus, the energy $\mathcal{E}(r,N)$ for the circulant graph $G(r,N)$ is given by the expression:
\begin{equation}\label{qe:energy_sum}
\mathcal{E}(r,N) = 2\sum_{k=0}^{N-1} \left| \sum_{m=1}^r \cos\left(\frac{2km\pi}{N}\right) \right|.
\end{equation}

\section{Closed formula for energy}

For small values of $r$, we can split the trigonometric sum \eqref{qe:energy_sum} defining $\mathcal{E}(r,N)$ into positive and negative parts:
$$\mathcal{E}(r,N) = 2\sum_{m=1}^r
\left(\sum_{u \left(r,\frac{2\pi k}{N}\right)\geq 0}\cos\left(\frac{2km\pi}{N}\right)-
\sum_{u \left(r,\frac{2\pi k}{N}\right)< 0}\cos\left(\frac{2km\pi}{N}\right)
\right).$$
For each $r$ we have a finite number of trigonometric sums along arithmetic sequences. 
Those sums can be computed explicitly, by means of the following elementary Lemma on trigonometric sums.

\begin{lemma}\label{lem:trig_sum}
Let $a_k = a_0 + rk$ be an arithmetic sequence of real numbers with $r\neq 0$. The following identities hold:
$$\sum_{k=0}^n \cos(a_k) = \frac{\sin(a_n+r/2)-\sin(a_0 - r/2)}{2\sin(r/2)},$$
$$\sum_{k=0}^n \sin(a_k) = \frac{\cos(a_0-r/2)-\cos(a_n + r/2)}{2\sin(r/2)}.$$
\end{lemma}

\begin{proof}
Let us take $z_k = e^{ia_k}$. We have 
$2\cos(a_k) = z_k + z_k^{-1}$ and $2i\sin(a_k) = z_k - z_k^{-1}$. Each sum split as the addition of two geometric sums that are explicitly computed. By undoing the same change of variables, we obtain the above identities.   
\end{proof}

For $r=1$ we obtain a closed formula for the energy of the cycle:
\begin{equation}\label{eq:en_1}
\mathcal{E}(1,N) = 4\frac{\sin\left( 
\frac{\pi}{N}\left( 2 \lfloor\frac{N}{4} \rfloor + 1 \right)
\right)}{\sin \left(\frac{\pi}{N}\right)}
\end{equation}

In the case $r=2$, we take into account:
$$\left|\cos(x) + \cos(2x)\right| = 
\begin{cases}  
\cos(x)+\cos(2x) &\mbox{si } \frac{-\pi}{3}\leq x \leq \frac{\pi}{3}  \\ 
-\cos(x)-\cos(2x) &\mbox{si } \frac{\pi}{3}\leq x \leq \frac{5\pi}{3} 
\end{cases}$$
and thus we obtain:
\begin{equation}\label{eq:en_2}
\mathcal{E}(2,N) = 4\left[
\frac{
\sin\left(\frac{\pi}{N}\left(2\lfloor\frac{N}{6}\rfloor+1\right)\right)
}{\sin\left(\frac{\pi}{N}\right)} + 
\frac{
\sin\left(\frac{2\pi}{N}\left(2\lfloor\frac{N}{6}\rfloor+1\right)\right)
}{\sin\left(\frac{2\pi}{N}\right)}
\right].
\end{equation}

\begin{figure}[h] 
  \centering
  \includegraphics[width=2.67in]{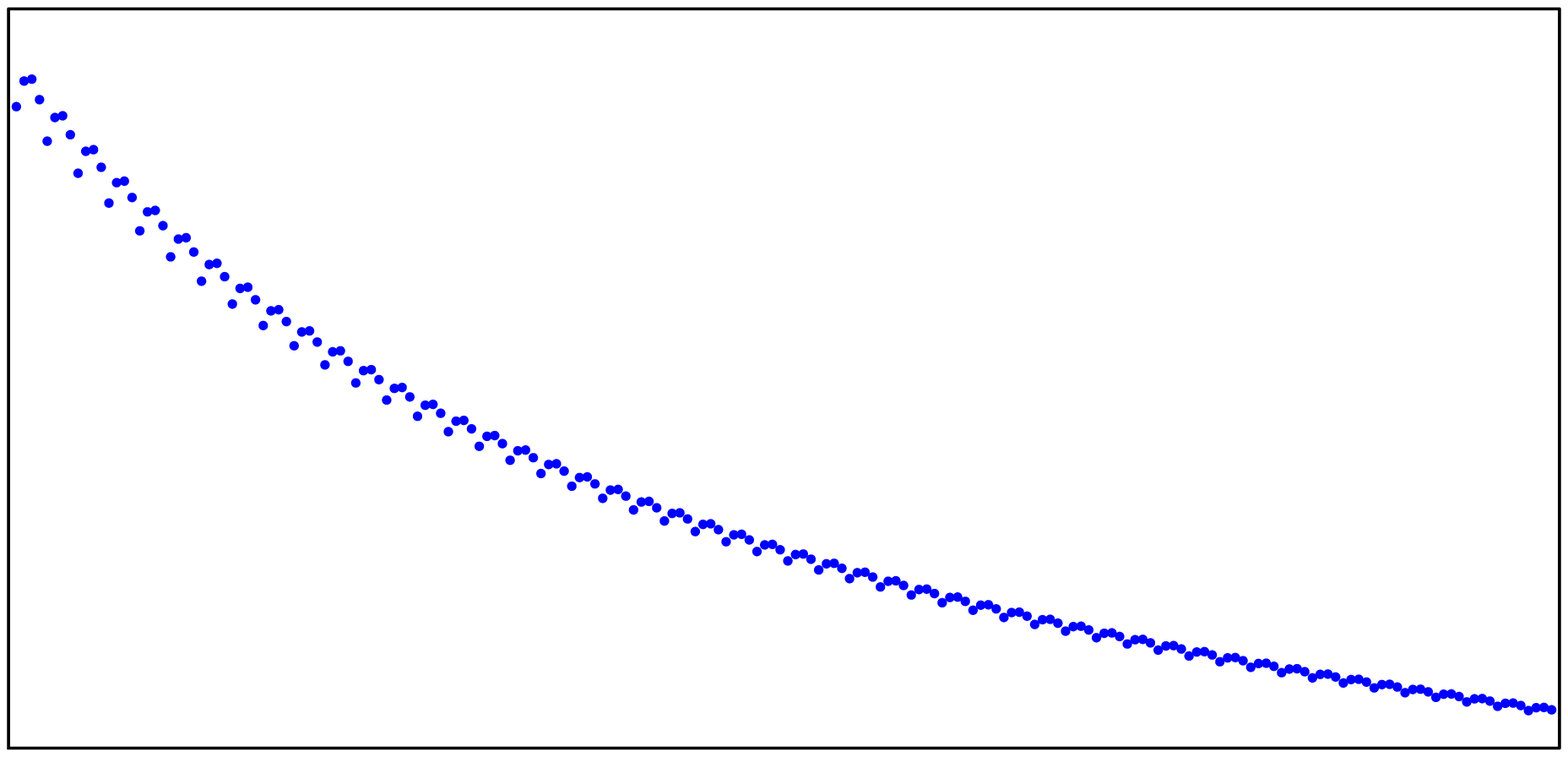} \includegraphics[width=2.67in]{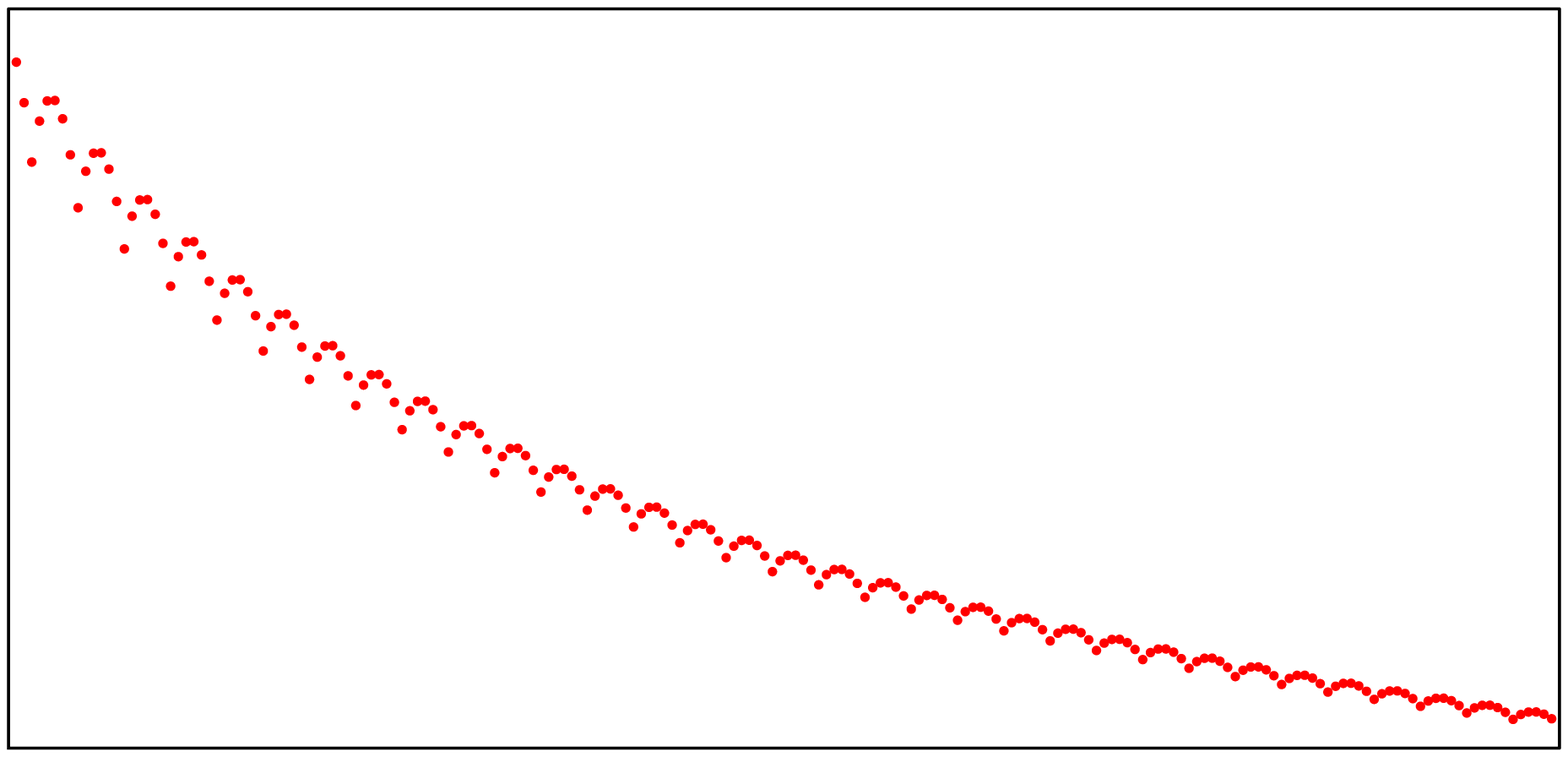}
  \caption[]{This picture illustrates the arithmetic behavior of the energy function. (Left) $x$-axis $N$, from $100$ to $300$. $y$-axis $\frac{\mathcal E(1,N)}{N-1}$ from $1.277$ to $1.287$. (Right) $x$-axis $n$, from $100$ to $300$. $y$-axis $\frac{\mathcal E(2,N)}{N-1}$ from $1.659$ to $1.672$.} 
\label{fi:energy_2}
\end{figure}

Figure \ref{fi:energy_2} illustrates the modular behavior (mod $4$ and mod $6$ respectively) of the  energy which is easily seen in formulae  \eqref{eq:en_1} and \eqref{eq:en_2}. Analogous formula, of growing complexity, may be computed for other values of $r$.

\medskip

By means of Lemma \ref{lem:trig_sum} we can also compute explicitly the energy of the graphs $G(r,2r+2)$ 
obtaining:
$$\mathcal{E}(r,2r+2) = 2r +  \sum_{k=1}^{2r+1} 
\left| \frac{\sin\left(k\pi - \frac{k\pi}{2r+2}\right)}
{\sin\left(\frac{k\pi}{2r+2}\right)}-1\right| = $$
$$= 2r + \sum_{k=1}^{2r+1} \left|(-1)^{k+1}-1\right| = 4r$$

Let us recall that the graph $G(r,2r+2)$ is isomorphic to $K_{2r+2}-M$; the graph
obtained by deleting a perfect matching from a complete graph of even order $2r+2$. It is a regular graph of degree $d = 2r$. It is well know that the energy of a regular graph of degree $d$ is equal or greater than $2d$ (see \cite[pag. 77]{Gutman2011}). This trivial lower bound is reached for the complete graph $K_{d+1}$. The family $K_{2r+2}-M$ gives us another example of regular graphs minimizing energy within each family of regular graphs of fixed even degree. 

\begin{theorem}
Let us consider the the graph $K_{2N}-M$ obtained by deleting the edges of a perfect matching of a complete graph of index $2N$.
\begin{itemize}
\item[(a)] $\mathcal E(K_{2N}-M) = 4N-4$ and thus $K_{2N}-M$ is non-hyperenergetic.
\item[(b)] $K_{2N}-M$ minimizes the energy in the family of regular graphs of degree $2N-2$.
\end{itemize}
\end{theorem}

\section{Asymptotic behaviour}

A direct consequence of Theorem 2 in \cite{Shparlinski2006} is that the limit $lim_{N\to\infty} \frac{\mathcal{E}(r,N)}{N-1}$ exists and it is bigger than
$4 \log(2r)/\pi^3$. Here we compute it explicitly.

\begin{theorem}
For each $r>0$ the limit:
$$lim_{N\to\infty} \frac{\mathcal{E}(r,N)}{N-1} = I_r$$
has value:
$$I_r = \frac{1}{\pi}
\int_0^\pi \left|
\frac{\sin\left(\left(\frac{1}{2}+r\right)\theta\right)}
{\sin\left(\frac{\theta}{2}\right)}-1
\right|d\theta. $$
\end{theorem}

Note that the integral $I_r$ is the $L_1$-norm of $D_r(\theta)-1$ where $D_r(\theta)$ is the Dirichlet kernel. From the triangular inequality we obtain:
$$ L_r - 1 \leq I_r \leq L_r + 1$$
where $L_r$ is the Lebesgue constant (see \cite{Zygmund} page 67).

We may evaluate analytically integral. Equation 
$$u(r,\theta)=0$$
can be solved analytically. It yields
$$\theta = \frac{2m\pi}{r}, \quad \theta = \frac{(2m+1)\pi}{r+1},
\quad m\in \mathbb Z.$$
Between $0$ and $\pi$ we obtain the partition:
$$0, \frac{\pi}{r+1}, \frac{2\pi}{r}, \frac{3\pi}{r+1},
\frac{4\pi}{r},\ldots, \pi.$$
The function $u(r,\theta)$ is positive in the first interval, negative in the second interval, and so on. This allows us to evaluate the integral obtaining, 
$$I_r = \frac{4}{\pi}\sum_{k=1}^r \sum_{m=0}^{\lfloor\frac{r}{2}\rfloor}  
\frac{\sin\left( \frac{(2m+1)k\pi}{r+1}\right) - \sin \left( \frac{2mk\pi}{r}\right)}{k}$$
and finally by application of Lemma \ref{lem:trig_sum}: 
$$I_r = 
\begin{cases}  
\displaystyle
\frac{2}{\pi} \sum_{k=1}^r \frac{1}{k}\left[
\frac{1-\cos\left(\frac{r k \pi}{r+1}\right)}{\sin \left(\frac{k\pi}{r+1}\right)} +
\frac{\cos\left(\frac{k \pi}{r}\right)- (-1)^k}{\sin \left(\frac{k\pi}{r}\right)} 
\right] &\mbox{ if $r$ odd,} \\
 \displaystyle
\frac{2}{\pi} \sum_{k=1}^r \frac{1}{k}\left[
\frac{1-(-1)^k}{\sin \left(\frac{k\pi}{r+1}\right)} +
\frac{\cos\left(\frac{k \pi}{r}\right)- \cos\left(\frac{(r+1) k \pi}{r}\right)}{\sin \left(\frac{k\pi}{r}\right)} 
\right] &\mbox{ if $r$ even.}
\end{cases}
$$
The first values can be computed explicitly obtaining, 
$$I_1 = \frac{4}{\pi}, \quad I_2 = \frac{3\sqrt{3}}{\pi}, \quad
I_3 = \frac{16\sqrt{2}-3\sqrt{3}}{3\pi}$$
and some approximate values are:
$$I_4 \simeq 1.985\ldots,\quad I_5 \simeq 2.087\ldots, \quad 
  I_6 \simeq 2.170\ldots$$
In particular we have that $I_r < 2$ for $r=1,2,3,4$ and $I_r>2$ for $r\geq 5$. It follows the following result.

\begin{theorem}\label{th:final}
The following statements hold:
\begin{itemize}
\item[(a)] For $r\leq 4$ there is a finite number of hyperenergetic graphs of the form $G(r,N)$.
\item[(b)] For each $r\geq 5$ there is a finite number on non-hyperenergetic graphs of the form $G(r,N)$. An example is given by $G(r,2r+2)$.
\end{itemize}
\end{theorem}

\begin{figure}[h]
  \centering
  \includegraphics[width=2.67in]{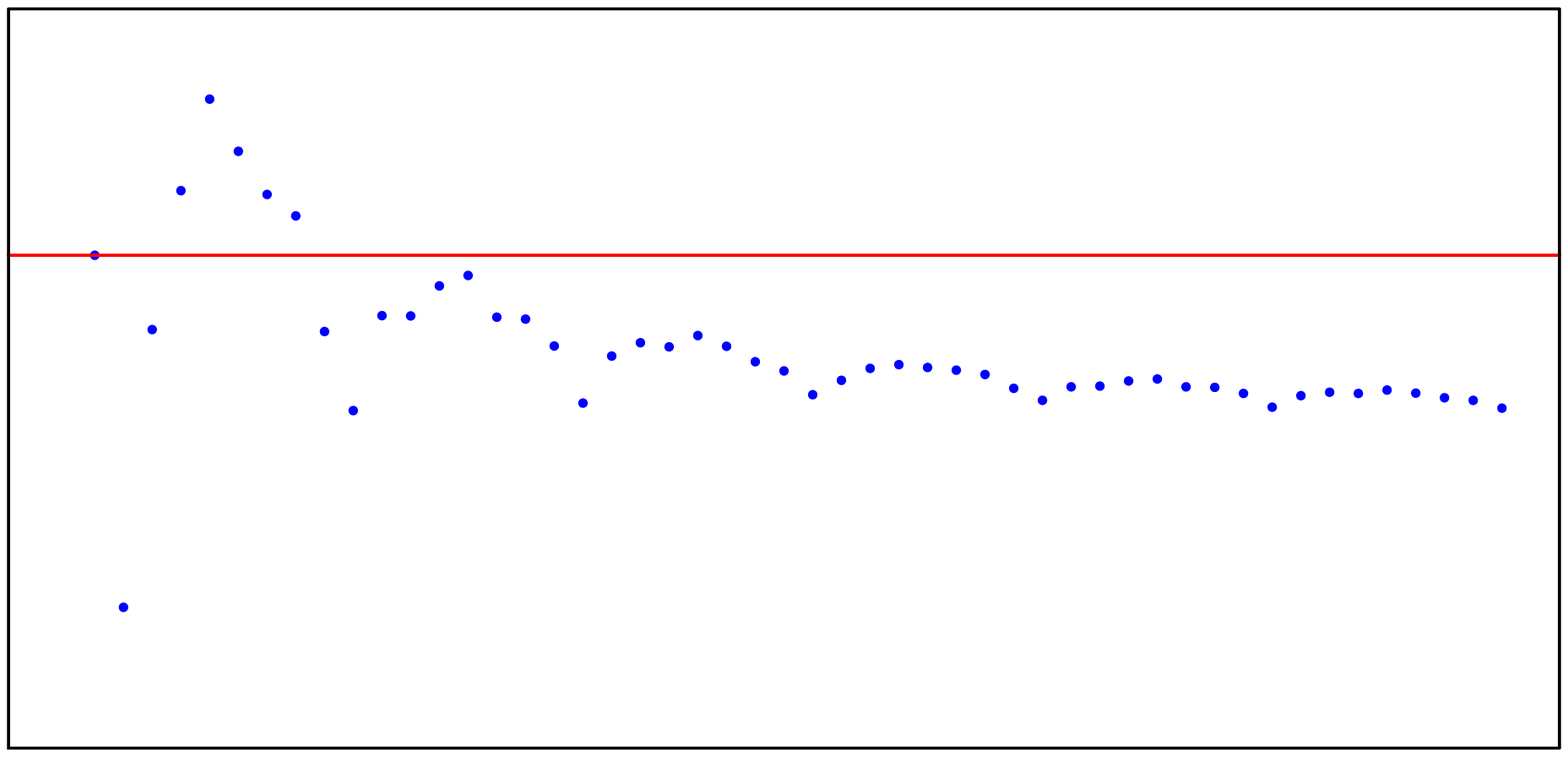} \includegraphics[width=2.67in]{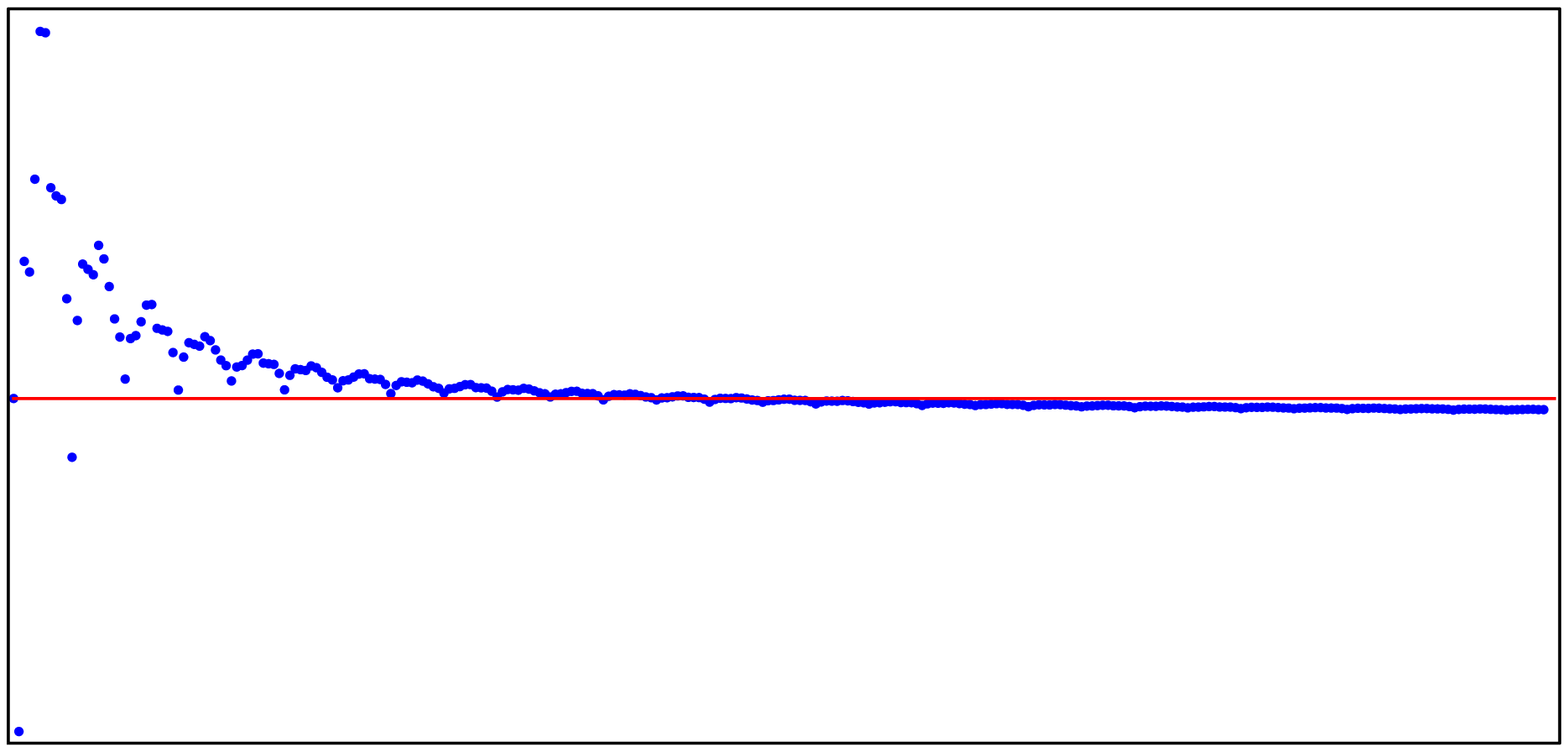} 
  \caption[]{(Left) $x$-axis $N$, from $7$ to $60$. $y$-axis $\frac{\mathcal E(3,N)}{N-1}$ from $1.6$ to $2.2$. The horizontal line corresponds to $y=2$. Dots over the horizontal line represent exceptional hyperenergetic graphs in the family $G(3,N)$. (Right) $x$-axis $N$, from $9$ to $300$. $y$-axis $\frac{\mathcal E(4,N)}{N-1}$ from $1.77$ to $2.26$. The horizontal line corresponds to $y=2$. Dots over the horizontal line represent exceptional hyperenergetic graphs in the family $G(4,N)$.} \label{figura_G4n}
\end{figure}

A numerical exploration (see fig. \ref{figura_G4n}) allows us to find all the hyperenergetic graphs of the form $G(r,N)$ with $r\leq 4$.

\begin{enumerate}
\item[(a)] There is no hyperenergetic graph of the form $G(1,N)$ or $G(2,N)$.
\item[(b)] There are five hyperenergectic graphs of the form $G(3,N)$ with $N = 12,13,14,15,16$.
\item[(c)] There is much bigger but finite family of hyperenergetic graphs of the form $G(4,N)$, ilustrated by fig. \ref{figura_G4n} (right). 
\end{enumerate}

With respect to Theorem \ref{th:final} (b), we conjecture that the only non-hyperenergetic graph of the form $G(r,N)$ with $r\geq 5$ is the graph $G(r,2r+2)$. 

\section*{Acknowledgements}

We want to thank Juan Pablo Rada for introducing us to the topic of graph energy, and sharing his experience with us. The first author acknowledges Universidad Nacional de Colombia research grant HERMES-27984. The second author  acknowledges the support of Universidad de Antioquia.





\end{document}